\documentclass[pdflatex,sn-mathphys-ay]{sn-jnl}

\usepackage{graphicx}
\usepackage{amsmath,amssymb,amsfonts}
\usepackage{mathrsfs}
\usepackage{tikz}
\usetikzlibrary{decorations.pathreplacing,automata,calc,positioning,arrows,shapes.geometric}

\theoremstyle{thmstyleone}
\newtheorem{theorem}{Theorem}
\newtheorem{corollary}[theorem]{Corollary}
\newtheorem{lemma}[theorem]{Lemma}
\newtheorem{proposition}[theorem]{Proposition}

\theoremstyle{thmstyletwo}

\theoremstyle{thmstylethree}

\newtheorem{question}{Question}

\begin{document}

\title[Connected Coloring Games on Corona Products]{Connected Coloring Games on Corona Products of Graphs}

\author*[1,2]{\fnm{Da\v sa} \sur{Mesari\v c \v Stesl}}\email{Dasa.Stesl@fri.uni-lj.si}

\affil*[1]{\orgdiv{Faculty of Computer and Information Science}, \orgname{University of Ljubljana}, \orgaddress{\street{Ve\v{c}na pot 113}, \postcode{1000}, \city{Ljubljana}, \country{Slovenia}}}

\affil[2]{\orgdiv{Faculty of Natural Sciences and Mathematics}, \orgname{University of Maribor}, \orgaddress{\street{Koro\v{s}ka cesta 160}, \postcode{2000}, \city{Maribor}, \country{Slovenia}}}

\abstract{We study the connected coloring game on corona products of graphs. For connected graphs $G$ and $H$, we prove the sharp upper bound $\chi_{cg}(G\circ H)\leq\max\{\Delta(G),\Delta(H)+1\}+1$ and establish a further bound using the Bob-start ordinary coloring game on the second factor. We show that the difference between the Alice-start and Bob-start connected game chromatic numbers can be arbitrarily large in either direction. We further prove that a connected subgraph may have arbitrarily larger connected game chromatic number than its ambient connected graph, and that $\chi_{cg}(G\circ H)$ cannot be bounded in terms of $\chi_{cg}(G)$ and $\chi_{cg}(H)$ alone, even when both factors have connected game chromatic number $2$. We also obtain corona products for which the difference $\chi_g-\chi_{cg}$ is arbitrarily large and construct infinitely many connected graphs with $\chi_g<\chi_{cg}$. Finally, for corona products of trees, we prove the universal upper bound $5$, characterize those attaining the value $3$, give structural sufficient conditions for the value $4$, and exhibit a broad family attaining the value $5$.}



\maketitle

\section{Introduction}
\label{sec:introduction}

The graph coloring game is a two-player game in which the first player's goal is to produce a proper vertex coloring of the entire graph, while the second player's goal is to prevent this from happening. The game was introduced by Steven Brams in the context of map coloring; it was presented by Gardner in 1981 and later discussed in a graph-theoretic framework by Bodlaender \citep{gardner-1981,bodlaender-1991}. Given a graph $G$ and a fixed set of colors, two players, named Alice and Bob, alternately take turns coloring previously uncolored vertices, always maintaining a proper coloring. Alice wins if the entire graph is eventually colored, while Bob wins if it becomes impossible to color some uncolored vertex. The least number of colors for which Alice has a winning strategy is the \emph{game chromatic number}, denoted by $\chi_g(G)$.

The game chromatic number has been extensively studied for many classes of graphs. A result of Faigle, Kern, Kierstead and Trotter states that every forest $F$ satisfies $\chi_g(F)\leq4$ \citep{faigle}. Further work considers exact values and bounds on the game chromatic numbers of planar and outerplanar graphs, partial $k$-trees, and other sparse graph classes \citep{kierstead-trotter-1994,dinski-zhu-1999,guan-zhu-1999,zhu-planar-1999,zhu-ppt-2000,kierstead-2000,zhu-activation-2008}. \cite{costa-2020} showed that the game is also computationally difficult, more precisely, that the corresponding decision problem is PSPACE-complete. For the main results concerning the classical map-coloring game, as well as for fundamental techniques in this field, see \cite{map-coloring-2007}.

Graph colorings of graph products represent a well-researched area and have a wide range of applications. The ordinary game chromatic number has been studied for Cartesian products \citep{bartnicki-cartesian-2008,zhu-cartesian-2008,sia-2009} and for strong products \citep{enomoto-2023}. Corona products have also been addressed in the ordinary coloring game. \cite{bokhary-2018} obtained exact values for several corona families, including $P_m\circ K_1$ and $P_m\circ C_n$, while \cite{alagammai-2020} determined a general upper bound and exact values for further families of corona products.

\cite{charpentier-2020} introduced the \emph{connected coloring game}, which is played according to the rules of the ordinary coloring game, with the additional requirement that the colored vertices at each stage of the game induce a connected subgraph. The minimum number of colors for which Alice has a winning strategy in this game is the \emph{connected game chromatic number}, denoted by $\chi_{cg}(G)$. It turned out that the connectivity requirement leads to behavior that differs significantly from that of the ordinary game. One of the main phenomena is the non-monotonicity of the connected coloring game with respect to the number of available colors. In particular, it was proved that every nontrivial connected bipartite graph has connected game chromatic number $2$, although for every $k\geq3$ there are connected bipartite graphs on which Bob has a winning strategy in the connected game played with $k$ colors \citep{charpentier-2020}. Furthermore, \cite{bradshaw-2023} studied the related connected game coloring number and resolved the question concerning outerplanar graphs. \cite{hollom-2024} later showed that the connectivity restriction can favor Bob by constructing a graph $G$ with $\chi_g(G)=4$ and $\chi_{cg}(G)=5$. Finally, it was proved that both the Alice-start and Bob-start versions of the connected coloring game are PSPACE-complete \citep{lima-2022}.

In this paper, we study the connected coloring game on corona products. The corona $G\circ H$ was introduced by \cite{frucht-harary-1970}. It consists of one copy of $G$ and a private copy $H^u$ of $H$ for every $u\in V(G)$, with $u$ adjacent to every vertex of $H^u$. This construction has a significant impact on the connected coloring game. Indeed, if the game takes place on the central graph, a copy $H^u$ cannot be reached before the vertex $u$ is colored. Moreover, once $u$ is colored, all vertices of the copy $H^u$ become accessible in terms of connectivity. This naturally divides the game into global play on the central graph and local play within the attached copies.

We first focus on general corona products. We prove the sharp degree bound $\chi_{cg}(G\circ H)\leq\max\{\Delta(G),\Delta(H)+1\}+1$. We also derive an additional bound in terms of the number of available colors by using the Bob-start ordinary coloring game on the second factor; this formulation avoids relying on any monotonicity assumption for the connected coloring game. We then determine the connected game chromatic number of $P_m\circ Q_n$, where $Q_n=K_{n,n}-\mathcal M$ and $\mathcal M$ is a perfect matching. In addition, we consider the Bob-start analogue and show that changing the starting player can alter the required number of colors by an arbitrarily large amount in either direction.

The same family gives several further consequences. The parameter $\chi_{cg}$ is highly non-monotone with respect to connected subgraphs: the difference between the connected game chromatic number of a connected subgraph and that of the ambient connected graph can be arbitrarily large. Moreover, $\chi_{cg}(G\circ H)$ cannot be bounded in terms of $\chi_{cg}(G)$ and $\chi_{cg}(H)$ alone; both factors can have connected game chromatic number $2$, while the corona product has arbitrarily large connected game chromatic number. We also compare the connected and ordinary coloring games. Within the class of corona products, the difference $\chi_g-\chi_{cg}$ can be arbitrarily large. In the opposite direction, by extending Hollom's example, we obtain infinitely many pairwise nonisomorphic connected graphs $G_k$ with $\chi_g(G_k)=4<5=\chi_{cg}(G_k)$.

In the second part of the paper, we focus on corona products of trees. We use the modified, expanded and reduced coloring games introduced by \cite{dllrt-2015} in their study of the game chromatic number of forests. Distinguishing between trees with game chromatic number $3$ and $4$ remains a central structural problem and has motivated a series of further results \citep{fddg-2019,furtado-2023,delrio-2026}. We prove that $\chi_{cg}(T_1\circ T_2)\leq5$ for all trees $T_1$ and $T_2$. Moreover, we give a complete characterization of corona products attaining the value $3$ when $|V(T_2)|\geq2$. This happens precisely when $T_2$ is a star, or when $T_2=P_4$ and $|V(T_1)|$ is odd. Next, we derive sufficient structural conditions for four colors from the reduced coloring game. Finally, we prove that the bound $5$ is attained by a broad family: if $|V(T_1)|\geq2$ and $T_2$ contains $H_d=\operatorname{cat}(0,2,2,2,2,0)$ as a connected subtree, then $\chi_{cg}(T_1\circ T_2)=5$.

The paper is organized as follows. In Section~\ref{sec:preliminaries}, we fix notation and recall the ordinary coloring game, its Alice-start and Bob-start variants, the connected coloring game, and the corona product. We then state general results for corona products, including two upper bounds, starting-player effects, subgraph phenomena, and comparisons between $\chi_g$ and $\chi_{cg}$. Section~\ref{sec:trees} is devoted to corona products of trees. More precisely, it develops transfer principles for auxiliary games, provides a characterization of corona products attaining the value $3$, states sufficient conditions for four colors, and describes families of graphs that require five colors. In Section~\ref{sec:conclusion}, we conclude with several open questions suggested by these results.

\section{Notation and preliminaries}
\label{sec:preliminaries}

All graphs in this paper are finite, simple, and undirected. For a graph $G$, we denote its vertex and edge sets by $V(G)$ and $E(G)$, respectively. The degree of a vertex $v\in V(G)$ is denoted by $d_G(v)$, and the maximum degree of $G$ by $\Delta(G)$. We denote the distance between vertices $u$ and $v$ by $d_G(u,v)$, and the diameter of $G$ by $\operatorname{diam}(G)$. We use the notation $\chi(G)$ and $\omega(G)$ for the chromatic number and the clique number of $G$, respectively.

We denote the complete graph, path, star, and complete bipartite graph by $K_n$, $P_n$, $K_{1,n}$, and $K_{m,n}$, respectively. The graph obtained from $K_{n,n}$ by deleting the edges of a perfect matching $\mathcal M$ is denoted by $Q_n=K_{n,n}-\mathcal M$.
The caterpillar with spine $v_1\ldots v_s$, in which exactly $k_i$ leaves are adjacent to $v_i$ for nonnegative integers $k_1,\ldots,k_s$, is denoted by $\operatorname{cat}(k_1,\ldots,k_s)$.
Let $G$ be a graph and let $C$ be a set of $k$ colors. The ordinary $k$-coloring game on $G$ is a two-player game between Alice and Bob, who alternately color previously uncolored vertices with colors from $C$. Alice moves first, and each move must preserve a proper coloring. Alice wins if all vertices of $G$ are eventually colored, whereas Bob wins as soon as some uncolored vertex has no available color. The least integer $k$ for which Alice has a winning strategy is called the \emph{game chromatic number} of $G$ and is denoted by $\chi_g(G)$ \citep{bodlaender-1991}.

We also consider the Bob-start version of the ordinary coloring game. The rules are the same as in the Alice-start game, except that Bob moves first. The least number of colors for which Alice has a winning strategy in the Bob-start coloring game on $G$ is denoted by $\chi_g^B(G)$.

From now on, assume that $G$ is connected and consider the \emph{connected coloring game}, introduced by \cite{charpentier-2020}. The same rules apply, with the additional requirement that the colored vertices induce a connected subgraph after every move. Thus, except for the first move, a vertex is playable precisely when it is uncolored and has a colored neighbor. The least number of colors for which Alice has a winning strategy is called the \emph{connected game chromatic number} of $G$ and is denoted by $\chi_{cg}(G)$.

The Bob-start connected coloring game is defined analogously, with the only difference that Bob moves first. The corresponding parameter is denoted by $\chi_{cg}^B(G)$.

We will repeatedly use the fact, proved by \cite{charpentier-2020}, that every nontrivial connected bipartite graph $G$ satisfies $\chi_{cg}(G)=2$. Note also that if Alice has a winning strategy in the connected coloring game on a connected bipartite graph $G$ when $k$ colors are available, this does not necessarily imply that she has a winning strategy when more colors are available \citep{charpentier-2020}. %Consequently, whenever the precise palette size matters, we formulate our statements for a fixed number of colors rather than relying on monotonicity.
For graphs $G$ and $H$, the \emph{corona product} $G\circ H$ is obtained from one copy of $G$ and $|V(G)|$ pairwise disjoint copies of $H$ by joining each vertex $u\in V(G)$ to every vertex of one corresponding copy of $H$ \citep{frucht-harary-1970}. We call the original copy of $G$ the \emph{central graph} and denote the copy of $H$ attached to $u$ by $H^u$.

If $x\in V(H^u)$ corresponds to $x_0\in V(H)$, then $d_{G\circ H}(x)=d_H(x_0)+1$, whereas $d_{G\circ H}(u)=d_G(u)+|V(H)|$.
Assume that the connected coloring game is played on the corona product $G\circ H$. We will use two basic observations throughout the paper. First, once the central vertex $u$ is colored, every vertex of $H^u$ is playable with respect to the connectivity condition. Second, if the connected coloring game starts in the central graph, a vertex of $H^u$ can be colored only after $u$ itself has been colored.
In the next section, we use three variants of the coloring game introduced by \cite{dllrt-2015}. We recall two of them here.% the reduced coloring game and the associated reduced graph are recalled when they are used.

The \emph{$k$-modified coloring game} ($k$-MCG) is the ordinary $k$-coloring game in which Bob starts the game and is allowed to pass. The advantage of this game is that moves made by Bob outside a chosen local subgraph may be interpreted as passes in the corresponding local game. \cite{dllrt-2015} proved, in particular, that Alice wins the $4$-MCG on every initially uncolored forest.

The \emph{$k$-expanded coloring game} ($k$-ECG) is the ordinary $k$-coloring game in which Alice may, instead of coloring a vertex, pass or add a single already colored leaf to the current forest. This facilitates Bob's local strategies. Namely, Alice's move outside a connected induced subforest can be represented locally either by a pass or, when the move is adjacent to the subforest, by the addition of a colored leaf \citep{dllrt-2015}.

We also use the well-known bound of \cite{faigle}, which states that every forest $F$ satisfies $\chi_g(F)\leq4$. In addition, the MCG, ECG, and reduced-game methods of \cite{dllrt-2015} will be used to distinguish between forests requiring two, three, or four colors.

\subsection{General upper bounds}

We first prove two general upper bounds. The first of them is sharp.

\begin{theorem}
\label{thm:general-degree-bound}
Let $G$ and $H$ be connected graphs. Then
\begin{center}
$\chi_{cg}(G\circ H)\leq \max\{\Delta(G),\Delta(H)+1\}+1.$
\end{center}
Moreover, the bound is sharp.
\end{theorem}

\begin{proof}
Let $k=\max\{\Delta(G),\Delta(H)+1\}+1$, and suppose that Alice and Bob play the connected coloring game on $G\circ H$ with $k$ colors. We show that Alice wins the game.

Alice starts by coloring an arbitrary vertex of the central copy of $G$. Observe that, for any $v\in V(G)$, no vertex of $H^v$ can be colored before $v$ itself is colored, due to the connectivity requirement. Indeed, the only edges joining $H^v$ to the rest of the corona product $G\circ H$ are the edges between $v$ and the vertices of $H^v$.

Let $v\in V(G)$ be an uncolored playable vertex. Since $v$ is uncolored, the copy $H^v$ is still completely uncolored. Hence every colored neighbor of $v$ lies in the central copy of $G$, and therefore $v$ has at most $\Delta(G)$ colored neighbors. Since $k\geq \Delta(G)+1$, at least one color is available at $v$.

Now let $x$ be an uncolored vertex of a copy $H^v$. The vertex $x$ is adjacent to $v$ and to at most $\Delta(H)$ vertices within $H^v$. Hence $x$ has at most $\Delta(H)+1$ colored neighbors. Since $k\geq \Delta(H)+2$, at least one color is available at $x$.

Therefore, every uncolored playable vertex can be properly colored throughout the game. Since $G\circ H$ is connected, whenever uncolored vertices remain, at least one of them is adjacent to the colored subgraph and is therefore playable. Thus, the game cannot end before all vertices have been colored, and Alice wins with $k$ colors.

To prove that the bound is sharp, take $G=H=K_n$. The above bound gives $\chi_{cg}(K_n\circ K_n)\leq n+1$. On the other hand, for every vertex $v$ of the central copy of $K_n$, the vertex $v$ together with the vertices of $K_n^v$ induces a clique of order $n+1$. Consequently, $\chi_{cg}(K_n\circ K_n)\geq n+1$, and equality follows.
\end{proof}

One of the most intriguing phenomena in coloring games concerns monotonicity with respect to the number of colors available in the game. The question of whether a winning strategy for Alice with $k$ colors in the ordinary coloring game also guarantees her a win when more than $k$ colors are available still remains open \citep{charpentier-2020,hollom-2024}. To avoid relying on any monotonicity assumption for the connected coloring game, we state the following result directly in terms of the number of available colors.

Recall that $\chi_g^B(H)$ denotes the minimum number of colors for which Alice has a winning strategy in the ordinary coloring game on $H$ when Bob makes the first move.

\begin{theorem}
\label{thm:odd-even-transfer}
Let $G$ and $H$ be connected graphs such that $|V(G)|$ is odd and $|V(H)|$ is even, and let $k\geq 2$. Suppose that Alice has a winning strategy in the connected $k$-coloring game on $G$ and a winning strategy in the Bob-start ordinary $(k-1)$-coloring game on $H$. Then Alice has a winning strategy in the connected $k$-coloring game on $G\circ H$.
\end{theorem}

\begin{proof}
Assume that Alice and Bob play the connected $k$-coloring game on $G\circ H$. Alice's strategy is as follows. She starts the game in the central copy of $G$ and follows her winning strategy for the connected $k$-coloring game on $G$. Whenever Bob colors a vertex of $G$, Alice responds in $G$ according to this strategy. Whenever Bob colors a vertex in a copy $H^v$, for some $v\in V(G)$, Alice responds in the same copy according to her winning strategy for the Bob-start ordinary $(k-1)$-coloring game on $H$, using the $k-1$ colors different from $c(v)$.

We first justify the local strategy. A copy $H^v$ cannot be entered before $v$ is colored. Once $v$ is colored, the connectivity requirement imposes no further restriction inside $H^v$, since every vertex of $H^v$ is adjacent to the colored vertex $v$. Moreover, the color $c(v)$ is forbidden at every vertex of $H^v$. Therefore, the game restricted to $H^v$ is precisely a Bob-start ordinary coloring game on $H$ with the $k-1$ colors different from $c(v)$.

Since Bob is the first player to move in every entered copy and $|V(H)|$ is even, Alice can always make her response in the same copy. Indeed, after each of Alice's local responses, an even number of vertices of that copy has been colored. Hence, whenever Bob next plays in that copy, either the copy is already completely colored or at least two vertices remain uncolored. Thus, Bob can never color the last vertex of a copy, and Alice makes the last move in every entered copy.

Consequently, all moves in the attached copies occur in Bob--Alice pairs and therefore do not affect the alternation of moves in the central copy of $G$. The game restricted to $G$ thus proceeds as a connected $k$-coloring game in which Alice moves first. Since $|V(G)|$ is odd, Alice also makes the last move in $G$. At that point it is Bob's turn, and hence Alice is never forced to be the first player to enter a previously untouched copy of $H$.

We conclude that Alice can simultaneously follow her winning strategy for the connected $k$-coloring game on $G$ and, in every entered copy of $H$, her winning strategy for the Bob-start ordinary $(k-1)$-coloring game on $H$. Therefore, she wins the connected $k$-coloring game on $G\circ H$.
\end{proof}

\begin{corollary}
\label{cor:odd-even-equal}
Let $G$ and $H$ be connected graphs such that $|V(G)|$ is odd and $|V(H)|$ is even. If $\chi_{cg}(G)=\chi_g^B(H)+1=k$, then $\chi_{cg}(G\circ H)\leq k$.
\end{corollary}

\subsection{Alice-start versus Bob-start connected coloring games on corona products}

In this subsection, we show that the difference between the connected game chromatic numbers of the Alice-start and Bob-start games can be arbitrarily large in either direction.
Recall that we denote by $Q_n$ the graph $K_{n,n}-\mathcal M$, where $\mathcal M$ is a perfect matching. For the $i$-th copy $Q_n^i$ in a corona product, let $A_i=\{a_1^i,\ldots,a_n^i\}$ and $B_i=\{b_1^i,\ldots,b_n^i\}$ be its partite sets, where $a_j^ib_j^i\notin E(Q_n^i)$ for every $j$.

\begin{proposition}
\label{prop:path-Qn}
For all integers $m,n\geq 2$,
\begin{center}
$\chi_{cg}(P_m\circ Q_n)=
\begin{cases}
n+1, & \text{if $m$ is even},\\
3, & \text{if $m$ is odd}.
\end{cases}$
\end{center}
\end{proposition}

\begin{proof}
Denote the vertices of the central path by $x_1,\ldots,x_m$, and let $Q_n^i$ be the copy attached to $x_i$.

First suppose that $m$ is even. In order to prove the lower bound, we consider a connected coloring game played with only $n$ colors. Until the central path is completely colored, Bob's winning strategy relies on the following pairing rule. Whenever Alice colors a vertex of $P_m$ and there remains an uncolored vertex of $P_m$, Bob responds by coloring an uncolored vertex of $P_m$ adjacent to the colored segment. Whenever Alice colors a vertex in an attached copy $Q_n^i$ for any $i\in\{1,2,\ldots,m\}$, Bob next makes a legal move in the same currently accessible part of that copy. Note that, if no such response is possible while uncolored vertices remain there, then the coloring of that copy cannot be completed and Bob has already won.

We first prove that Alice has to be the first player to color a vertex of $P_m$. Clearly, this is true if the game starts on the path. Therefore, suppose that the game starts in a copy $Q_n^i$. Before $x_i$ is colored, the connectivity condition restricts the game to the connected component of $Q_n^i$ in which Alice made her first move. For $n\geq 3$, the graph $Q_n$ is connected and has $2n$ vertices, and for $n=2$, each component of $Q_2$ is an edge. In either case, the relevant component has even order. Since Bob answers each move of Alice in that component, either Alice colors $x_i$ before the component is exhausted, or Bob makes the last move in the component and Alice is then forced to color $x_i$. This proves that Alice makes the first move on the central path.

From this point on, every move of Alice outside the path is paired with a response of Bob in the same copy, while every move of Alice on the unfinished path is paired with a response of Bob on the path. Hence the moves restricted to $P_m$ alternate, with Alice moving first. Since $m$ is even, Bob colors the last vertex of the path; denote it by $x_j$. The copy $Q_n^j$ is untouched at this moment. Indeed, it was inaccessible before $x_j$ was colored, unless the game started in $Q_n^j$. In the latter case, $x_j$ would have been the first, rather than the last, vertex of the path to be colored.

At this moment, it is Alice's turn. Every copy that has already been entered contains an even number of colored vertices, and therefore also an even number of uncolored vertices. If Alice continues in such a copy, Bob can keep responding there until either he wins or he makes the last move in that copy. Since at least one attached copy is still untouched, after finitely many such completions Alice is forced to be the first player to enter an untouched copy $Q_n^\ell$ whose central vertex $x_\ell$ is already colored.

Bob now plays along the missing edges of $\mathcal{M}$. More precisely, whenever Alice colors $a_k^\ell$ or $b_k^\ell$, Bob colors the nonadjacent counterpart $b_k^\ell$ or $a_k^\ell$, respectively, with the same color. This response is legal. Indeed, every previously completed pair has one common color. Alice is not able to use the color of a previously completed pair on a new pair, since her chosen vertex is adjacent to the opposite vertex of that previous pair. Consequently, distinct pairs along the missing edges of $\mathcal{M}$ must receive distinct colors.

Thus the $n$ pairs $\{a_k^\ell,b_k^\ell\}$ require $n$ distinct colors, and each of these colors must be different from $c(x_\ell)$. Thus $\chi_{cg}(P_m\circ Q_n)\geq n+1$.

For the reverse inequality, assume that $n+1$ colors are available in the connected coloring game played on $P_m\circ Q_n$. Alice first colors the vertex $x_1$. As long as the central path is not completely colored, whenever it is her turn she colors an uncolored vertex of $P_m$ adjacent to the colored segment. Such a vertex has exactly one colored neighbor on the path. Moreover, its attached copy of $Q_n$ is still untouched because of the connectivity condition of the game. Thus, a legal color is available for it. Every vertex in an attached copy $Q_n^i$ has degree $n$ in the corona product. Namely, it has $n-1$ neighbors in the opposite partite set and one neighbor $x_i$ in the central path. Therefore, no vertex of an attached copy can become uncolorable when the game is played with $n+1$ colors. We conclude that Alice wins the game with $n+1$ colors and $\chi_{cg}(P_m\circ Q_n)\leq n+1$.

Now assume that $m$ is odd. Since the considered corona contains a triangle, at least three colors are necessary. We show that three colors also suffice.

We consider the connected coloring game on $P_m\circ Q_n$ with three colors and show that Alice has a winning strategy in such a game. She starts by playing $x_1$. Whenever Bob plays on the unfinished central path, Alice extends the colored segment of the path. Whenever Bob colors a vertex $a_k^i$ or $b_k^i$ in an attached copy $Q_n^i$, Alice colors the vertex in the opposite bipartite set along the missing edge of $\mathcal{M}$ with the unique color different from both $c(x_i)$ and the color just used by Bob.

Every pair that is entered is therefore completed immediately. On the first such pair in $Q_n^i$, suppose Bob colors $a_k^i$ with color $c_1$ and Alice colors $b_k^i$ with color $c_2$. Then $c_1,c_2,c(x_i)$ are the three distinct colors. Every remaining vertex of $A_i$ can then receive $c_1$, since it is adjacent to $x_i$, colored with $c(x_i)$, and to $b_k^i$, colored with $c_2$. Similarly, every remaining vertex of $B_i$ can receive $c_2$, since it is adjacent to $x_i$ and to $a_k^i$, colored with $c_1$. This means that all later responses in the copy are legal and predetermined.

Moves in attached copies occur in Bob--Alice pairs and do not affect the alternation on $P_m$. Alice is the first player on the path and, since $m$ is odd, she also colors its last vertex. Hence Bob is the first player to enter every attached copy, so the preceding local strategy applies in each of them. Thus Alice wins the game with three colors.
\end{proof}

Recall that the invariant in the Bob-start connected version of the game is denoted by $\chi_{cg}^B$.

\begin{proposition}
\label{prop:Bob-path-Qn}
For all integers $m\geq 2$ and $n\geq 3$,
\begin{center}
$\chi_{cg}^B(P_m\circ Q_n)=
\begin{cases}
4, & \text{if $m$ is even},\\
n+1, & \text{if $m$ is odd}.
\end{cases}$
\end{center}
\end{proposition}

\begin{proof}
Suppose first that $m$ is even and that Alice and Bob play the Bob-start connected coloring game on $P_m\circ Q_n$. We show that four colors suffice.

Assume first that Bob starts the game on the central path. Alice responds to every move of Bob on the unfinished path by extending the colored segment of $P_m$. If Bob enters an attached copy $Q_n^i$ for any $i\in{1,2,\ldots,m}$, Alice responds in the same copy. When Bob makes his first move there, say at $a_k^i$ with color $c_1$, Alice responds by coloring $b_k^i$ with a color $c_2$ different from both $c_1$ and $c(x_i)$. From this moment on, color $c_1$ is available at every uncolored vertex of $A_i$, while color $c_2$ is available at every uncolored vertex of $B_i$. Therefore, Alice can answer every subsequent move of Bob in $Q_n^i$ by coloring the nonadjacent counterpart with the appropriate one of these two colors. In this way, Alice colors the last vertex of every entered copy, since $|V(Q_n)|$ is even. Moreover, all moves in attached copies occur in Bob--Alice pairs. The moves restricted to the even path also alternate between Bob and Alice, with Bob moving first, which means that Alice makes the last move on $P_m$. It is then Bob's turn, so he is the first player to enter every completely uncolored copy. Consequently, Alice can finish the game using four colors.

Next assume that Bob starts the game in a copy $Q_n^i$ for some $i\in{1,2,\ldots,m}$. By symmetry, suppose that he first colors the vertex $a_1^i$ with color $c_1$. Alice responds by coloring the central vertex $x_i$ with a color $c_2\neq c_1$. We call ${a_1^i,b_1^i}$ the initial pair. It is the only pair that is incomplete without having been opened by Bob after its central vertex has been colored. Alice leaves $b_1^i$ uncolored as long as possible and uses it as a parity correction. Whenever Bob plays on the unfinished central path, Alice extends the colored segment of the path, except in the single parity-correction step described below. We first describe Alice's local strategy in $Q_n^i$, which ensures that $b_1^i$ always remains colorable.

Before Bob ever colors a vertex of $A_i-{a_1^i}$, whenever he colors $b_{\ell}^i$ with $\ell\neq1$, Alice responds by coloring $a_{\ell}^i$ with color $c_1$. Alice's move is legal because every vertex of $B_i-{b_1^i}$ is adjacent to $a_1^i$ and thus cannot be colored with color $c_1$.

Suppose next that Bob colors a vertex $a_t^i\in A_i-{a_1^i}$. Among the two colors different from $c_1$ and $c_2$, Alice chooses a color $c_3$ different from $c(a_t^i)$ and colors $b_t^i$ with this color. At this moment, every still uncolored vertex of $A_i-{a_1^i,a_t^i}$ can receive color $c_1$. From now on, $c_3$ is forbidden at every still uncolored vertex of $A_i$, because all such vertices are adjacent to $b_t^i$. Alice can therefore maintain the following rule. If Bob colors $b_s^i$, she colors $a_s^i$ with $c_1$, while if Bob colors $a_s^i$, she colors $b_s^i$ with $c_3$. Both responses are legal. In particular, $c_3$ never appears on a neighbor of $b_1^i$, so $b_1^i$ remains colorable.

The same immediate pairing strategy is used in every other attached copy $Q_n^j$ once its central vertex has been colored. Indeed, after Bob opens the first missing pair, Alice colors its counterpart with a color different from both the color of the central vertex and the color just used by Bob. The two colors used on this first pair then remain available on the two respective partite sets.

If Bob colors $b_1^i$ before the central path is complete, Alice responds on the path. This gives Alice one additional unpaired move on $P_m$, besides her initial move at $x_i$. All subsequent moves on the path again occur in Bob--Alice pairs. Since $m$ is even, Alice then makes the last move on $P_m$.

Otherwise, suppose that Bob does not color $b_1^i$ before the central path is completely colored. Then, apart from Alice's initial move at $x_i$, the moves on $P_m$ occur in Bob--Alice pairs. Since $m$ is even, Bob makes the last move on the path. Alice then responds by coloring $b_1^i$. This move is legal by the invariant above.

After the initial pair has been completely colored, the initially entered copy $Q_n^i$ is safe. If $c(b_1^i)\neq c_1$, then $c_1$ is available at every still uncolored vertex of $A_i$. If the color $c_3$ has already been introduced, then $c_3$ is available at every still uncolored vertex of $B_i$, so the coloring can be completed using these two colors. If $c_3$ has not yet been introduced, then all colored vertices of $A_i$ have color $c_1$. Since $c(b_1^i)\notin\{c_1,c_2\}$, the color $c(b_1^i)$ is available at every still uncolored vertex of $B_i$. Hence the copy can also be completed in this case.

It remains to consider the case $c(b_1^i)=c_1$. This can occur only if no previously colored neighbor of $b_1^i$ has received color $c_1$. If $c_3$ has not yet been introduced, the two colors different from $c_1$ and $c_2$ can be used as permanent colors on the two partite sets. If $c_3$ has already been introduced, let $c_4$ be the fourth color. Then $c_4$ is available at every still uncolored vertex of $A_i$, while $c_3$ is available at every still uncolored vertex of $B_i$. Hence the coloring of the initial copy can always be completed.

In either case, once the path is complete and the initial pair has been completely colored, it is Bob's turn. Every untouched copy is thus entered first by Bob, and every already entered copy is safe. Therefore, Alice wins the Bob-start connected coloring game with four colors, and
$\chi_{cg}^B(P_m\circ Q_n)\leq 4$.

To prove the reverse inequality, assume that the Bob-start connected coloring game is played with only three colors. Bob first colors $a_1^1$ with color $1$. By connectivity, Alice can respond either by coloring $x_1$ or by coloring a vertex $b_k^1$ with $k\neq1$.

If Alice colors $x_1$, Bob responds by coloring $b_1^1$ also with color $1$. Every remaining vertex of $A_1$ and $B_1$ is adjacent to $x_1$ and to a vertex of color $1$. Consequently, the two partite sets would both have to use the single remaining color, although there are edges between them. Thus, the coloring cannot be completed with only three colors.

If Alice colors $b_k^1$ for some $k\neq1$, then $a_1^1$ and $b_k^1$ have different colors. Since $n\geq3$, Bob can choose $\ell\notin\{1,k\}$ and color $a_\ell^1$ with the third color. This move is connected through $b_k^1$. Now the uncolored vertex $x_1$ has neighbors in all three colors and is therefore uncolorable. Hence $\chi_{cg}^B(P_m\circ Q_n)\geq4$,
and together with the upper bound, equality follows. This concludes the proof for even $m$.

From this point on, assume that $m$ is odd. We first prove the upper bound. Suppose that Alice and Bob play the Bob-start connected coloring game with a set of $n+1$ colors. If Bob starts in an attached copy $Q_n^i$, Alice immediately colors its central vertex $x_i$. Otherwise, if Bob starts on the path, no special first response is needed. Thereafter, every uncolored vertex of $P_m$ that is available has exactly one colored neighbor on the path, and its own attached copy is completely uncolored. Hence it has an available color. Every vertex in an attached copy has degree $n$ in the corona product, so it also has an available color whenever it is played. Therefore, the game can always be completed, and $\chi_{cg}^B(P_m\circ Q_n)\leq n+1$.

In order to prove the lower bound, assume that only $n$ colors are available in the Bob-start connected coloring game on $P_m\circ Q_n$. Bob first colors the vertex $x_1$. Whenever Alice colors a vertex of the unfinished central path and an uncolored vertex of $P_m$ remains, Bob responds by coloring an uncolored vertex of $P_m$ adjacent to the colored segment. Whenever Alice plays in a copy $Q_n^i$ whose central vertex is already colored, Bob colors the opposite vertex along the corresponding missing edge of the perfect matching with the same color. We justify that such a move is legal in exactly the same way as in the proof of Proposition~\ref{prop:path-Qn}. Moreover, this strategy keeps an even number of vertices colored in every entered copy. If, before such a copy is completely colored, there exists an uncolored vertex that cannot be colored, then Bob has already won.

Thus, moves outside the path occur in Alice--Bob pairs, and the moves restricted to $P_m$ alternate with Bob moving first. Since $m$ is odd, Bob colors the last vertex of $P_m$. At that moment, it is Alice's turn, every previously entered copy has an even number of uncolored vertices, and at least one copy of $Q_n$ is still completely uncolored. By repeatedly responding to Alice in any previously entered copy, Bob either wins there or makes the last move in that copy. Hence Alice is eventually forced to make the first move in an untouched copy $Q_n^j$ whose central vertex is already colored.

Bob now applies the following strategy. Whenever Alice colors one endpoint of a missing edge of the perfect matching $\mathcal{M}$, Bob responds by coloring its other endpoint with the same color. Distinct pairs corresponding to the edges of $\mathcal{M}$ must receive distinct colors, while the color of $x_j$ is forbidden at every vertex of $Q_n^j$. Therefore, $n$ distinct pair-colors in addition to $c(x_j)$ would be required, which is impossible if the game is played with only $n$ colors. This gives $\chi_{cg}^B(P_m\circ Q_n)\geq n+1$. Altogether, $\chi_{cg}^B(P_m\circ Q_n)=n+1$ if $m$ is odd.
\end{proof}

The two preceding propositions show that the identity of the starting player can change the required number of colors by an arbitrarily large amount in either direction.

\begin{theorem}
\label{thm:start-player-gap}
For every positive integer $k$, there exist connected graphs $G$ and $H$ such that
$\chi_{cg}(G)-\chi_{cg}^B(G)=k$
and
$\chi_{cg}^B(H)-\chi_{cg}(H)=k$.
\end{theorem}

\begin{proof}
Let $k$ be a positive integer. For any even $m\geq 2$, set $G=P_m\circ Q_{k+3}$. By Propositions~\ref{prop:path-Qn} and~\ref{prop:Bob-path-Qn}, we have $\chi_{cg}(G)=k+4$ and $\chi_{cg}^B(G)=4$, which implies $\chi_{cg}(G)-\chi_{cg}^B(G)=k+4-4=k$.

For any odd $m\geq 3$, set $H=P_m\circ Q_{k+2}$. Propositions~\ref{prop:path-Qn} and~\ref{prop:Bob-path-Qn} imply that $\chi_{cg}(H)=3$ and $\chi_{cg}^B(H)=k+3$. Therefore, $\chi_{cg}^B(H)-\chi_{cg}(H)=k+3-3=k$.
\end{proof}

\subsection{Subgraphs and comparison of the connected game chromatic number of a corona product with those of its factors}

In order to show that the connected game chromatic number of a connected subgraph can be arbitrarily larger than that of the whole graph, we first prove the following lemma.

\begin{lemma}
\label{lem:P2-Knn}
For every $n\geq 2$, $\chi_{cg}(P_2\circ K_{n,n})=4$.
\end{lemma}

\begin{proof}
Denote the central vertices of $P_2\circ K_{n,n}$ by $x_1$ and $x_2$, and let the copy $K_{n,n}^i$ attached to $x_i$ have partite sets $A_i$ and $B_i$.

We first show that Bob wins the connected coloring game on $P_2\circ K_{n,n}$ if the game is played with three colors. First consider the case when Alice starts the game at $x_i$. Bob responds by coloring the other central vertex. Both attached copies are then completely uncolored. Hence Alice must be the first player to enter one of them, say $K_{n,n}^j$. Assume that she colors $v\in A_j$. Bob colors a different vertex $u\in A_j$ with the third color, namely the color different from $c(x_j)$ and $c(v)$. Such a vertex exists since $n\geq 2$. Every vertex of $B_j$ is adjacent to $x_j$, $u$, and $v$, which have three distinct colors. Thus no vertex of $B_j$ can be colored.

Next suppose that Alice starts the game in a copy $K_{n,n}^i$ while $x_i$ is still uncolored. Before $x_i$ is colored, Bob responds to each move of Alice in this copy by coloring an uncolored vertex in the opposite partite set. Such a vertex exists because, after each response of Bob, the same number of vertices has been colored in the two partite sets. Moreover, Bob's move satisfies the connectivity condition, since his chosen vertex is adjacent to the vertex just colored by Alice. If no available color exists for such a response, then an uncolored vertex is already uncolorable, and Bob has already won.

Suppose first that Alice colors $x_i$ before the copy is completely colored. Bob responds by coloring the other central vertex $x_j$. If Alice enters the still completely uncolored copy $K_{n,n}^j$, Bob wins by the strategy described in the preceding paragraph. Otherwise, Bob continues to respond to Alice's moves in $K_{n,n}^i$ by coloring a vertex in the opposite partite set. Since the number of uncolored vertices remaining in $K_{n,n}^i$ is even, Bob makes the last move in this copy. Consequently, Alice is forced to begin coloring $K_{n,n}^j$, and Bob again wins by the strategy described above.

Finally, suppose that Alice does not color $x_i$ before the initially entered copy is completely colored. Then Bob makes the last move in that copy. Alice must color $x_i$ on her next move, after which Bob colors $x_j$. Alice is then forced to enter the completely uncolored copy $K_{n,n}^j$, and the same argument applies. Therefore,
$\chi_{cg}(P_2\circ K_{n,n})\geq4$.

For the upper bound, suppose that the connected coloring game on $P_2\circ K_{n,n}$ is played with four colors. Alice makes her first move at $x_1$. If Bob colors $x_2$, Alice enters either of the two attached copies. Otherwise, by the connectivity requirement, Bob must play in $K_{n,n}^1$, and Alice responds by coloring $x_2$. Thus, after Alice's second move, both central vertices are colored.

From this point on, Alice uses the following local strategy. Consider a copy $K_{n,n}^i$ whose central vertex $x_i$ is already colored. As soon as both partite sets contain a colored vertex, say $a\in A_i$ and $b\in B_i$, the copy is safe. Indeed, $x_i,a,b$ form a triangle, and hence their colors are pairwise distinct. Since every vertex of $B_i$ is adjacent to $a$, no vertex of $B_i$ can subsequently receive the color $c(a)$. Moreover, $c(a)\neq c(x_i)$. Consequently, the color $c(a)$ remains available at every uncolored vertex of $A_i$. Analogously, $c(b)$ remains available at every uncolored vertex of $B_i$. Hence no vertex of this copy can ever become uncolorable.

It therefore remains to consider a copy in which only one of the two partite sets, say $A_i$, contains colored vertices. If it is Alice's turn, she colors a vertex of $B_i$ with a color different from $c(x_i)$ and from all colors currently used on $A_i$. If Bob instead colors a vertex of $B_i$, the copy immediately becomes safe.

The only remaining possibility is that Bob colors another vertex of $A_i$ before any vertex of $B_i$ has been colored. At this moment, $A_i$ contains at most two colored vertices: the first colored vertex of the copy and the vertex colored by Bob in his preceding move. Therefore, at most three colors are forbidden at a vertex of $B_i$, namely $c(x_i)$ and the colors used on these two vertices. Alice can thus color any vertex of $B_i$ with the remaining fourth color, making the copy safe.

Alice applies this strategy whenever one of the attached copies is not yet safe. An unfinished copy in which only one partite set has been entered cannot become dangerous while play takes place elsewhere, since no additional colors are introduced into that copy. Moreover, if Bob plays again in its already entered partite set, that set contains at most two colored vertices before Alice can respond by entering the opposite partite set. Once both partite sets have been entered, the permanent available colors described above ensure that no vertex of the copy can become uncolorable.

Therefore Alice can complete the coloring of the whole graph using four colors. Hence $\chi_{cg}(P_2\circ K_{n,n})\leq 4$.
\end{proof}

\begin{theorem}
\label{thm:connected-subgraph-gap}
For every positive integer $k$, there exist a connected graph $G$ and a connected subgraph $H$ of $G$ such that
$\chi_{cg}(H)-\chi_{cg}(G)\geq k$.
\end{theorem}

\begin{proof}
Let $G=P_2\circ K_{k+3,k+3}$, and let $H=P_2\circ Q_{k+3}$. Clearly $H$ is a connected subgraph of $G$. By Lemma~\ref{lem:P2-Knn}, $\chi_{cg}(G)=4$, while Proposition~\ref{prop:path-Qn} implies that $\chi_{cg}(H)=k+3+1=k+4$. We conclude that $\chi_{cg}(H)-\chi_{cg}(G)=k$.
\end{proof}

The same family also shows that the connected game chromatic number of a corona product cannot be bounded in terms of the connected game chromatic numbers of its factors.

\begin{theorem}
\label{thm:factors-two-unbounded}
For every positive integer $k$, there exist connected graphs $G$ and $H$ such that
$\chi_{cg}(G)=\chi_{cg}(H)=2$
and
$\chi_{cg}(G\circ H)-\max\{\chi_{cg}(G),\chi_{cg}(H)\}\geq k$.
\end{theorem}

\begin{proof}
\cite{charpentier-2020} proved that the connected game chromatic number of every nontrivial connected bipartite graph equals $2$. Set $G=P_2$, and let $H=Q_{k+2}$. Since $k+2\geq 3$, both $G$ and $H$ are connected bipartite graphs, so we obtain $\chi_{cg}(G)=\chi_{cg}(H)=2$.

By Proposition~\ref{prop:path-Qn}, $\chi_{cg}(G\circ H)=(k+2)+1=k+3$. Thus $\chi_{cg}(G\circ H)-\max\{\chi_{cg}(G),\chi_{cg}(H)\}=k+1\geq k$.
\end{proof}

\subsection{Comparison with the ordinary coloring game}

We first recover, within the class of corona products, the phenomenon that the ordinary game chromatic number may exceed the connected game chromatic number by an arbitrarily large amount.

\begin{lemma}
\label{lem:Qn-K1-game}
Let $n\geq 3$ and let $Q_n=K_{n,n}-\mathcal M$, where $\mathcal M$ is a perfect matching. Then $\chi_g(Q_n\circ K_1)\geq n$.
\end{lemma}

\begin{proof}
Let the partite sets of $Q_n$ be $A=\{x_1,\ldots,x_n\}$ and $B=\{y_1,\ldots,y_n\}$, where $x_i y_i\notin E(Q_n)$ for every $i$. Furthermore, let $x_i'$ and $y_i'$ denote the leaves attached to $x_i$ and $y_i$, respectively, in the corona $Q_n\circ K_1$.

Suppose that the ordinary coloring game is played with at most $n-1$ colors. Bob uses the following pairing strategy. For every $i$, he pairs $x_i$ with $y_i$ and $x_i'$ with $y_i'$. Whenever Alice colors the first vertex of one of these pairs, Bob responds by coloring the other vertex of the same pair with the same color.

We first show that Bob's response is always legal. Consider a central pair ${x_i,y_i}$ and suppose, without loss of generality, that Alice colors $x_i$. The vertices $x_i$ and $y_i$ are nonadjacent. Moreover, every previously colored central pair ${x_j,y_j}$ is monochromatic by Bob's strategy. Hence, if Alice's color on $x_i$ were already used on a colored central neighbor $x_j$ of $y_i$, then the paired vertex $y_j$ would have the same color. Since $i\neq j$, the vertex $y_j$ is adjacent to $x_i$, contradicting the legality of Alice's move.

It remains to consider the leaf $y_i'$, which is the only neighbor of $y_i$ outside $Q_n$. If $y_i'$ has already been colored, then $x_i'$ has also been colored, with the same color, by Bob's pairing strategy. Since $x_i'$ is adjacent to $x_i$, the color chosen by Alice for $x_i$ differs from the color of $x_i'$, and hence also from the color of $y_i'$. Thus Bob may legally color $y_i$ with the same color as $x_i$. The case in which Alice first colors $y_i$ is symmetric.

Now consider a pendant pair ${x_i',y_i'}$. Suppose, without loss of generality, that Alice colors $x_i'$. If neither $x_i$ nor $y_i$ has been colored, Bob's response at $y_i'$ is clearly legal. Otherwise, both $x_i$ and $y_i$ have already been colored and, by the pairing strategy, have the same color. Since Alice's move at $x_i'$ is legal, its color differs from $c(x_i)=c(y_i)$, and therefore Bob may use the same color on $y_i'$. Hence Bob can maintain the pairing strategy throughout the game.

In particular, for every $i$, the two central vertices $x_i$ and $y_i$ receive the same color. Moreover, distinct central pairs must receive distinct colors. Indeed, if $i\neq j$ and the pairs ${x_i,y_i}$ and ${x_j,y_j}$ received the same color, then the adjacent vertices $x_i$ and $y_j$ would have the same color.

Consequently, after $n-1$ central pairs have been colored, they use all of the available $n-1$ colors. Let ${x_i,y_i}$ be the remaining uncolored central pair. The vertex $x_i$ is adjacent to $y_j$ for every $j\neq i$, and these vertices carry all $n-1$ available colors. Hence $x_i$ is uncolorable. By symmetry, the same is true for $y_i$. Thus Bob wins the game with $n-1$ colors, and therefore $\chi_g(Q_n\circ K_1)\geq n$.
\end{proof}

Recall that \cite{charpentier-2020} proved that the connected game chromatic number of every connected bipartite graph equals $2$. Since $Q_n\circ K_1$ is connected and bipartite, we have $\chi_{cg}(Q_n\circ K_1)=2$. This immediately yields the following result.

\begin{theorem}
\label{thm:ordinary-connected-gap}
For every positive integer $k$, there exists a corona product graph $G$ such that $\chi_g(G)-\chi_{cg}(G)\geq k$.
\end{theorem}

\begin{proof}
Take $G=Q_{k+2}\circ K_1$. Since $k+2\geq 3$, the graph $Q_{k+2}$, and hence also $G$, is connected and bipartite. Therefore $\chi_{cg}(G)=2$. On the other hand, Lemma~\ref{lem:Qn-K1-game} gives $\chi_g(G)\geq k+2$. Consequently,
$\chi_g(G)-\chi_{cg}(G)\geq k$.
\end{proof}

\cite{hollom-2024} showed that the connectivity restriction may also favor Bob by constructing a graph $G$ with $\chi_g(G)=4$ and $\chi_{cg}(G)=5$. We now extend this phenomenon to an infinite family.

\begin{theorem}
\label{thm:hollom-family}
There exist infinitely many pairwise nonisomorphic connected graphs $G_k$, $k\geq 0$, such that $\chi_g(G_k)<\chi_{cg}(G_k)$.
\end{theorem}

\begin{proof}
Let $G$ be the graph constructed by Hollom and presented in Figure~2 of \cite{hollom-2024}, and denote by $v$ the vertex with label $1$ in that figure. Let $w$ denote the other vertex of degree $5$ in $G$, that is, the unlabeled vertex in the lower left part of Figure~2. Notice that every vertex of $G$ other than $v$ and $w$ has degree at most $4$. For $k\geq 0$, let $G_k$ be the graph obtained from $G$ by adding $k$ paths $vx_iy_i$, for $1\leq i\leq k$, where $x_i,y_i$ are new vertices. In particular, $G_0=G$.

First let Alice and Bob play the ordinary coloring game on $G_k$ when four colors are available. Hollom showed that Alice has a winning strategy in the game on $G$ played with four colors if she starts the game at $v$. Therefore, Alice first plays $v$ and follows her winning strategy on $G$. Whenever Bob colors a vertex in any of the added paths, Alice responds by coloring the other uncolored vertex of that path. Since that vertex has at most two colored neighbors, Alice's response is always legal. Such a Bob--Alice pair of moves on an added path does not change the position or the player to move in the game restricted to $G$. Therefore, Alice can continue with her winning strategy on $G$. Consequently, $\chi_g(G_k)\leq4$. Since $G$ contains $K_4$ as a subgraph, $K_4$ is also a subgraph of $G_k$, and thus $\chi_g(G_k)\geq4$.

Next consider the connected coloring game played with four colors on $G_k$ for any $k\geq0$. Since $\chi_{cg}(G)=5$, Bob has a winning strategy in the game played with four colors on $G$. Suppose first that Alice starts the game in $G$. Whenever Alice plays a vertex on an added path, she must first color $x_i$ because of the connectivity requirement of the game, and Bob responds by coloring $y_i$. Again, such an Alice--Bob pair of moves does not affect the colored subgraph of $G$ or the player to move there, so Bob can continue with his strategy restricted to $G$.

Otherwise, suppose that Alice makes the first move of the game in an added path. Bob responds by coloring the other vertex of the same path. Alice is then forced to color $v$ on her next move because of the connectivity condition. Indeed, $v$ is the only vertex outside the path adjacent to the colored set. Then it is Bob's turn, and the game restricted to $G$ is precisely the connected four-color game in which Alice has started at $v$. Therefore, Bob can again use his winning strategy on $G$. We conclude that Bob wins the game played with four colors and $\chi_{cg}(G_k)\geq5$.

Finally, we show that Alice has a winning strategy in the connected coloring game on $G_k$ with five colors. She starts by coloring $v$. If Bob colors $w$ on his first move, Alice makes any legal move. Otherwise, Alice colors $w$ on her second move. This move is legal, since $w$ is adjacent to $v$ and at this moment has at most two colored neighbors.

Once both $v$ and $w$ are colored, every uncolored vertex of $G_k$ has degree at most $4$. Indeed, every other vertex of $G$ has degree at most $4$, while the newly added vertices $x_i$ and $y_i$ have degrees $2$ and $1$, respectively. Hence, with five colors available, every uncolored playable vertex has an available color. Since $G_k$ is connected, as long as an uncolored vertex remains, at least one uncolored vertex is playable. Therefore, the game can always be completed, and $\chi_{cg}(G_k)\leq5$.

We conclude that $\chi_g(G_k)=4$ and $\chi_{cg}(G_k)=5$ for every $k\geq0$. In addition, the graphs $G_k$ are pairwise nonisomorphic since $|V(G_k)|=|V(G)|+2k$.
\end{proof}

\begin{figure}[ht]
\centering
\begin{tikzpicture}[
    vertex/.style={circle,draw,minimum size=7mm,inner sep=0pt}
]

% Hollom's graph G
\node[vertex] (a) at (0,2.2) {};
\node[vertex] (b) at (0,0.9) {};
\node[vertex] (w) at (-0.8,-0.2) {$w$};
\node[vertex] (r) at (2.0,2.2) {};
\node[vertex] (v) at (2.0,0.9) {$v$};
\node[vertex] (d) at (3.0,-0.2) {};
\node[vertex] (s) at (1.2,-1.2) {};

% Edges of Hollom's graph G
\draw (a)--(r);
\draw (a)--(b);
\draw (a)--(w);
\draw (a)--(v);
\draw (b)--(v);
\draw (w)--(v);
\draw (w)--(d);
\draw (w)--(s);
\draw (r)--(v);
\draw (r)--(d);
\draw (v)--(d);
\draw (d)--(s);

% Edge w--r routed completely outside the graph
\draw (w) -- (-1.8,-0.2) -- (-1.8,3.0) -- (2.0,3.0) -- (r);

% Added paths v x_i y_i
\node[vertex] (x1) at (4.4,2.0) {$x_1$};
\node[vertex] (y1) at (5.7,2.0) {$y_1$};

\node[vertex] (x2) at (4.4,1.2) {$x_2$};
\node[vertex] (y2) at (5.7,1.2) {$y_2$};

\node at (4.95,0.45) {$\vdots$};

\node[vertex] (xk) at (4.4,-0.4) {$x_k$};
\node[vertex] (yk) at (5.7,-0.4) {$y_k$};

\draw (v)--(x1)--(y1);
\draw (v)--(x2)--(y2);
\draw (v)--(xk)--(yk);

\node at (1.2,-1.8) {$G$};

\end{tikzpicture}

\caption{An infinite family of graphs $G_k$ such that $\chi_g(G_k)<\chi_{cg}(G_k)$.}
\label{fig:hollom-family}
\end{figure}
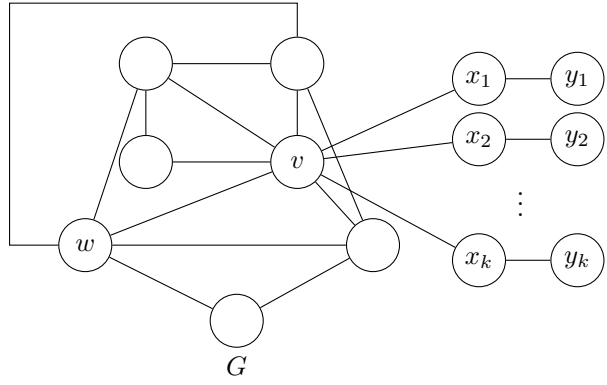

\section{Corona products of trees}
\label{sec:trees}

It is known that every forest $F$ satisfies $\chi_g(F)\leq 4$ \citep{faigle}.
In this section, we study the connected coloring game on corona products of two trees.
For $v\in V(T_1)$, we denote by $T_2^v$ the copy of $T_2$ attached to $v$ in $T_1\circ T_2$.

In Section~\ref{sec:preliminaries}, we recalled the modified and expanded coloring games of \cite{dllrt-2015}, which will be used repeatedly. The MCG is a useful tool for proving upper bounds, since a move of Bob outside a fixed local game can be interpreted as a pass. Similarly, the ECG is useful for proving lower bounds. Indeed, a move of Alice outside a fixed induced subtree can be represented locally either by a pass or by the addition of a colored leaf when it is adjacent to the subtree.

\subsection{Bounds via modified and expanded coloring games}

We first prove an auxiliary lemma that we will use several times later on.

\begin{lemma}
\label{lem:mcg-transfer}
Let $T_1$ and $T_2$ be trees, and let $k\geq 1$. If Alice has a winning strategy in the $k$-MCG on an uncolored $T_2$, then $\chi_{cg}(T_1\circ T_2)\leq k+1$.
\end{lemma}

\begin{proof}
Let $C$ denote a set of $k+1$ colors. Alice starts the game by coloring an arbitrary vertex of the central tree $T_1$.

Since the colored vertices of $T_1$ induce a connected subtree after every move of the game, every uncolored vertex $u\in V(T_1)$ that is adjacent to the colored subtree has exactly one colored neighbor in $T_1$. Moreover, while $u$ remains uncolored, no vertex of the copy $T_2^u$ can be colored because of the connectivity condition of the game. Thus every move of Alice on $T_1$ is legal.

Further, let $v\in V(T_1)$ be a fixed colored vertex. Clearly, the color $c(v)$ is not available at any vertex of $T_2^v$, and therefore the game in this copy is played with the remaining $k$ colors. Moreover, once $v$ has been colored, there are no further connectivity restrictions inside the copy $T_2^v$, since every vertex of $T_2^v$ is adjacent to $v$, which has already been colored.

Alice uses an independent winning $k$-MCG strategy in every copy $T_2^v$. If, after Bob's move in a copy $T_2^v$, there still exists an uncolored vertex in this copy, Alice responds in the same copy according to her winning $k$-MCG strategy. If Bob colors the last uncolored vertex of that copy, Alice needs no local response. Whenever it is Alice's move and no response in a copy is needed, she makes a move in the central tree $T_1$, if possible. Otherwise, if all vertices of $T_1$ are already colored, she chooses any copy $T_2^v$ containing an uncolored vertex and makes a move according to her $k$-MCG strategy in this copy, interpreting Bob's preceding local move as a pass. The same interpretation applies if Alice is the first player to enter a completely uncolored copy $T_2^v$.

Therefore, when we consider the moves in any fixed copy in their local order, they form a legal play of the $k$-MCG. Indeed, Bob is the first local player whenever he is the first to enter the copy. If Alice enters the copy first, Bob's first local move is interpreted as a pass. Likewise, every later unpaired move of Alice in that copy is interpreted as being preceded by a pass of Bob. Thus Alice's winning $k$-MCG strategy on $T_2$ implies that every attached copy can be completely colored using the described strategy. Since the coloring of the central tree can always be extended whenever it still contains an uncolored vertex, Alice is able to color all vertices of $T_1\circ T_2$ when $k+1$ colors are available.
\end{proof}
By Lemma~6 of \cite{dllrt-2015}, Alice has a winning strategy in the $4$-MCG on every uncolored forest. Thus, Lemma~\ref{lem:mcg-transfer} immediately gives the following universal upper bound.

\begin{theorem}
\label{thm:universal-five}
Let $T_1$ and $T_2$ be trees. Then $\chi_{cg}(T_1\circ T_2)\leq 5$.
\end{theorem}

For lower bounds the expanded coloring game will be used.

\begin{lemma}
\label{lem:ecg-transfer}
Let $T_1$ and $T_2$ be trees with $|V(T_1)|\geq 2$, and let $H$ be a connected subtree of $T_2$. If Bob has a winning strategy in the $k$-ECG on an uncolored $H$, then Bob also wins the connected coloring game on $T_1\circ T_2$ with $k+1$ colors.
\end{lemma}

\begin{proof}
Let $C$ denote a set of $k+1$ colors. Whenever a central vertex $w\in V(T_1)$ is colored, we consider the game on $H^w$ as being played with the color set $C\setminus\{c(w)\}$. Since every vertex of $H^w$ is adjacent to $w$, the color $c(w)$ is not available at any vertex of $H^w$. Thus the presence of $w$ imposes no additional restriction on the game played with the remaining $k$ colors on $H^w$.

Assume first that Alice starts the game at a vertex $u\in V(T_1)$. Bob considers the copy $H^u\subseteq T_2^u$ and interprets Alice's first move at $u$ as a pass in the $k$-ECG. He then uses his winning ECG strategy on $H^u$. If Alice colors a vertex of $H^u$, this is interpreted as her ordinary ECG move. If she colors a vertex of $T_2^u-H^u$ adjacent to $H^u$, Bob interprets that move as the addition of a colored leaf. Since $T_2$ is a tree and $H$ is connected, such a vertex is adjacent to exactly one vertex of $H^u$. Moreover, its color lies in $C\setminus\{c(u)\}$ since it is a neighbor of $u$. Every other move of Alice outside $H^u$ is interpreted as a pass. Clearly, Bob's moves in $H^u$ satisfy the connectivity condition, since every vertex of $H^u$ is adjacent to the already colored vertex $u$. This ensures the legality of the simulated ECG.

Next assume that Alice makes the first move of the game at a vertex of an attached copy $T_2^u$. Bob responds by coloring the vertex $u$ with a color different from the color used by Alice in her first move. Let $v$ be any neighbor of $u$ in $T_1$. Such a vertex exists because $|V(T_1)|\geq 2$. If Alice colors $v$ on her next move, then $T_2^v$ is still completely uncolored and it is Bob's turn. In this case Bob uses $H^v$ as his target and interprets Alice's move at $v$ as the initial pass of the ECG.

Otherwise, if Alice does not color $v$ in her second move, Bob colors $v$ on his next move. This move is legal since the copy $T_2^v$ is still completely uncolored. Indeed, no vertex of this copy can be colored before $v$ because of the connectivity condition of the game. Moreover, among the vertices of $T_1$, the only colored neighbor of $v$ is $u$. Since $T_1$ is a tree and the colored vertices induce a connected subgraph, no other neighbor of $v$ can have been colored before $v$. Therefore, only the color $c(u)$ is forbidden at $v$. It is then Alice's turn, and her next move is interpreted as the first action of the ECG on $H^v$. It is either a move in $H^v$, the addition of a colored leaf when she plays at a vertex of $T_2^v-H^v$ adjacent to $H^v$, or a pass otherwise.

In all considered cases Bob obtains a faithful simulation of his winning $k$-ECG strategy on a copy of $H$ with the $k$ colors different from the color of its central vertex. Hence, at some point in the game, the simulated ECG produces an uncolored vertex of $H^w$ at which none of the colors distinct from $c(w)$ is available. The remaining color $c(w)$ is not available either, since every vertex of $H^w$ is adjacent to $w$. Thus the obtained vertex is uncolorable in $T_1\circ T_2$, and Bob wins the connected coloring game played with $k+1$ colors.
\end{proof}
\cite[Lemmas~9 and~10]{dllrt-2015} showed that Bob has a winning strategy in the $2$-ECG on an uncolored $P_5$ and on the tree $P_4^+$ obtained from $P_4$ by attaching one additional leaf to an internal vertex. This implies the following consequence that will be useful in the continuation.

\begin{corollary}
\label{cor:nonstar-not-p4}
Let $T_1$ and $T_2$ be trees with $|V(T_1)|\geq 2$ and $|V(T_2)|\geq 2$. If $T_2$ is neither a star nor $P_4$, then $\chi_{cg}(T_1\circ T_2)\geq 4$.
\end{corollary}

\begin{proof}
First note that $T_2$ contains an edge, and therefore $T_1\circ T_2$ contains a triangle. Hence two colors do not suffice. It remains to show that Bob wins the connected coloring game on $T_1\circ T_2$ with three colors.

If $\operatorname{diam}(T_2)\geq 4$, then $T_2$ contains $P_5$ as a connected subtree, and the result follows from Lemma~\ref{lem:ecg-transfer} and Bob's winning $2$-ECG strategy on $P_5$.

If $\operatorname{diam}(T_2)=3$, then $T_2$ is a double star. Its two central vertices are adjacent and each has at least one leaf neighbor. Since $T_2\neq P_4$, at least one of the two central vertices has degree at least $3$. Hence, by taking one leaf adjacent to each central vertex and one additional leaf adjacent to the central vertex of degree at least $3$, we obtain $P_4^+$ as a connected subtree of $T_2$. Again, Lemma~\ref{lem:ecg-transfer} and Bob's winning $2$-ECG strategy on $P_4^+$ yield the result.

Finally, every tree with at least two vertices and diameter at most $2$ is a star, so there are no further cases to consider.
\end{proof}

\subsection{The corona products of trees with connected game chromatic number three}

In this subsection we characterize corona products of trees with the connected game chromatic number equal to three.

\begin{lemma}
\label{lem:stars}
Let $T$ be a tree and let $n\geq 1$. Then $\chi_{cg}(T\circ K_{1,n})=3$.
\end{lemma}

\begin{proof}
The lower bound follows from the fact that $T\circ K_{1,n}$ contains a triangle.

To prove the upper bound, assume that the connected coloring game is played with three colors. Alice starts by coloring an arbitrary vertex of $T$. Her only immediate-response rule is the following. If Bob is the first player to enter a copy $K_{1,n}^v$ and colors a leaf while the center of that star is still uncolored, Alice responds by coloring the center with the unique color different from $c(v)$ and the color used by Bob in his previous move. Clearly, this move satisfies the connectivity requirement and is legal.

In all remaining cases, Alice colors an uncolored vertex of $T$ adjacent to the colored subtree, provided such a vertex still exists. Notice that such a vertex has exactly one colored neighbor in $T$, while its own attached copy is still completely uncolored by the connectivity condition. Therefore, it always has an available color.

Otherwise, if the central tree is already completely colored, Alice has no prescribed response in the central tree. She may then play in any attached copy that is not yet completely colored. Whenever Alice is the first player to enter such a copy, she colors the center of the star with any color different from the color of its central vertex in $T$.

Once the center of an attached star is colored, every uncolored leaf has a unique available color, namely the color different from the colors of the central vertex $v$ and the center of the star. Consequently, every such leaf is always colorable. Moreover, by Alice's immediate-response rule, a copy in which a leaf is colored while the center is still uncolored can occur only immediately after Bob's move, and Alice colors the center at once.

Thus, using the described strategy, Alice can always make a legal move until all vertices of $T\circ K_{1,n}$ are colored. Hence $\chi_{cg}(T\circ K_{1,n})\leq 3$, and therefore $\chi_{cg}(T\circ K_{1,n})=3$.
\end{proof}
The path $P_4$ is the unique nonstar second factor that can still yield the value $3$ if the first factor contains at least two vertices, according to Corollary~\ref{cor:nonstar-not-p4}.

\begin{lemma}
\label{lem:p4}
Let $T_1$ be a tree. Then $\chi_{cg}(T_1\circ P_4)=3$ if and only if $|V(T_1)|$ is odd.
\end{lemma}

\begin{proof}
Since $T_1\circ P_4$ contains a triangle, at least three colors are necessary.

Suppose first that $|V(T_1)|$ is odd. Alice makes the first move of the game at an arbitrary vertex of $T_1$. Whenever Bob plays in $T_1$ and there still exists an uncolored vertex of $T_1$, Alice colors an uncolored vertex of $T_1$ adjacent to the colored subtree. Whenever Bob plays in an attached copy $P_4^v$, Alice responds by coloring a vertex at distance two from Bob's colored vertex in that copy with the same color.

The latter response is always legal and possible. Indeed, when Bob makes the first move in $P_4^v$, Alice can color a vertex at distance two from his vertex. The two remaining uncolored vertices of the copy are again at distance two, so the same response is available after Bob's second move in that copy. Note that vertices at distance two are nonadjacent. Moreover, every color used in the copy is different from $c(v)$.

In this way every Bob's move in an attached copy is paired with an immediate Alice's answer in the same copy. The sequence of moves restricted to $T_1$ thus alternates, with Alice playing first. Since $|V(T_1)|$ is odd, Alice makes the last move in $T_1$. Consequently, Bob starts the game in every attached copy, and by the above described pairing strategy every copy can be completely colored.

It remains to consider the case when $|V(T_1)|$ is even. Notice that if $v\in V(T_1)$ is colored, while $P_4^v$ is completely uncolored, and Alice has to make the first move in $P_4^v$, then Bob has a strategy for an immediate win. Indeed, if Alice colors $x\in V(P_4^v)$, Bob plays a vertex $z$ at distance two from $x$ with the third color, different from $c(v)$ and $c(x)$. If $y$ is the vertex between $x$ and $z$, then its neighborhood contains three different colors used on $v,x,z$, and hence $y$ is uncolorable.

If Alice starts the game in $T_1$, Bob responds in $T_1$ whenever possible. Alice cannot be the first player to enter an untouched attached copy without losing by the above observation. Since $|V(T_1)|$ is even, Bob makes the last move in $T_1$, after which Alice is forced to enter an untouched attached copy first, which again implies Bob's win.

Finally, assume that Alice makes the first move of the game in a copy $P_4^u$. Denote by $S$ the subgraph induced by $V(T_1)\cup V(P_4^u)$. Whenever Alice plays in $S$, Bob also plays in $S$. Such a response is always possible unless Bob has already won. Indeed, the graph $S$ is connected, and before it is completely colored, it always contains an uncolored vertex of $S$ adjacent to the colored set. If none of these available boundary vertices has a legal color, then an uncolored vertex has already become uncolorable and Bob has won. If Alice makes the first move in any other attached copy, Bob wins by the strategy described above. Thus, as long as Alice has not lost, all moves are made in $S$ until $S$ is completely colored. Since $|V(S)|=|V(T_1)|+4$ is even and Alice starts the game in $S$, the last move in $S$ is made by Bob. Since $|V(T_1)|\geq 2$, there exists at least one completely uncolored attached copy, in which Alice is forced to play first. Bob again wins with the same strategy.

Thus $\chi_{cg}(T_1\circ P_4)=3$ exactly when $|V(T_1)|$ is odd.
\end{proof}
We need to consider separately the case in which the first factor is $P_1$.

\begin{lemma}
\label{lem:p1-characterization}
Let $T$ be a tree with $|V(T)|\geq 2$. Then $\chi_{cg}(P_1\circ T)=3$ if and only if $T$ is a star or $T=P_4$.
\end{lemma}

\begin{proof}
We denote by $v$ the unique vertex of $P_1$. The sufficiency follows from Lemmas~\ref{lem:stars} and~\ref{lem:p4}. Therefore, we only need to prove the converse.

Suppose that $T$ is neither a star nor $P_4$. Then $\operatorname{diam}(T)\geq 3$. We are looking for Bob's winning strategy in the connected coloring game played with three colors.

Assume first that Alice starts the game at a vertex $x\in V(T)$. Let $y$ be a neighbor of $x$ such that, after $x$ and $y$ have been colored, there is a path $abc$ in $T$ with $a\in\{x,y\}$ and $b,c$ still uncolored. Such a choice is always possible. If the eccentricity of $x$ is at least $3$, take a shortest path $xybc$ of length at least $3$ beginning at $x$ and set $a=y$. If the eccentricity of $x$ is $2$, take a path $xbc$ of length $2$. In this case $x$ must have a neighbor $y\neq b$: otherwise $x$ would be a leaf whose unique neighbor is adjacent to every other vertex, forcing $\operatorname{diam}(T)\leq 2$, a contradiction. We then set $a=x$. Bob colors $y$ with a color different from $c(x)$.

We claim that Alice must next color $v$. Assume, on the contrary, that she does not. Due to the connectivity requirement of the game, she makes another move in $T$. If she uses the third color, then vertex $v$ already has all three colors in its neighborhood and is therefore uncolorable. Consequently, Alice needs to use one of the two colors $c(x)$ and $c(y)$. If she colors $b$, Bob responds by coloring $c$ with the third color. Notice that such a move is legal, and then $v$ again has all three colors in its neighborhood. If Alice colors any other vertex, then $b$ remains uncolored and Bob responds by coloring $b$ with the third color. Such a move of Bob is legal because, apart from $a$, no neighbor of $b$ can yet be colored without passing through $b$ in the tree. Again $v$ becomes uncolorable. Hence Alice is forced to play $v$ using the third color.

In this case Bob answers by coloring $c$ with the color different from $c(v)$ and $c(a)$. Since $c$ is adjacent to the colored vertex $v$ but not to $a$, this move is legal. Then $b$ becomes uncolorable since it is adjacent to $v$, $a$, and $c$, which received three distinct colors.

It remains to consider the case in which Alice plays $v$ first.

If $\operatorname{diam}(T)\geq 4$, denote by $x_1x_2x_3x_4x_5$ a path in $T$. Bob colors $x_3$ with a color different from $c(v)$. After Alice's next move, at least one of the two endpoints $x_1,x_5$, together with all its neighbors in $T$, remains uncolored. Indeed, the closed neighborhoods of $x_1$ and $x_5$ in $T$ are disjoint, since $d_T(x_1,x_5)=4$, while Alice colors only one vertex. Without loss of generality, assume that $x_1$ and all its neighbors in $T$ remain uncolored. Then Bob colors $x_1$ with the third color, different from $c(v)$ and $c(x_3)$. This move is legal and concludes the game, since $x_2$ has all three colors on its neighbors $v,x_1,x_3$ and therefore becomes uncolorable.

Finally, let $\operatorname{diam}(T)=3$. Then $T$ is a double star. Since $T\neq P_4$, one of its two central vertices, say $y$, has at least two leaf neighbors $b_1,b_2$. Denote by $x$ the other central vertex and by $a$ a leaf adjacent to $x$. Without loss of generality, assume that $c(v)=1$. Then Bob colors $b_1$ with color $2$.

If Alice colors $y$, then $y$ must receive color $3$, and Bob colors $a$ with color $2$, making $x$ uncolorable. If Alice colors $x$, Bob colors $b_2$ with color $3$, making $y$ uncolorable. If Alice colors a leaf adjacent to $x$, then either she uses color $3$, in which case Bob colors $b_2$ with color $3$, or she uses color $2$, in which case Bob colors $x$ with color $3$; in either case $y$ becomes uncolorable. If Alice colors another leaf adjacent to $y$, then color $3$ makes $y$ uncolorable immediately, while after color $2$ Bob colors $x$ with color $3$ and again makes $y$ uncolorable.

We conclude that Bob wins with three colors whenever $T$ is neither a star nor $P_4$.
\end{proof}

Next, we give a complete characterization for corona products of trees whose second factor is nontrivial.

\begin{theorem}
\label{thm:three-characterization}
Let $T_1$ and $T_2$ be trees with $|V(T_2)|\geq 2$. Then $\chi_{cg}(T_1\circ T_2)=3$
if and only if one of the following holds:
\begin{enumerate}
    \item $T_2$ is a star;
    \item $T_2=P_4$ and $|V(T_1)|$ is odd.
\end{enumerate}
\end{theorem}

\begin{proof}
Since the graph $T_1\circ T_2$ contains a triangle, it follows that $\chi_{cg}(T_1\circ T_2)\geq 3$.

The two conditions are sufficient by Lemmas~\ref{lem:stars} and~\ref{lem:p4}.

For necessity, assume first that $T_1=P_1$. Lemma~\ref{lem:p1-characterization} implies that the only possibilities are that $T_2$ is a star or $T_2=P_4$, and in the latter case $|V(T_1)|=1$ is odd.

Next, suppose that $|V(T_1)|\geq 2$. If $T_2$ is neither a star nor $P_4$, then Corollary~\ref{cor:nonstar-not-p4} implies that $\chi_{cg}(T_1\circ T_2)\geq 4$. If $T_2=P_4$, Lemma~\ref{lem:p4} states that the value is $3$ exactly when $|V(T_1)|$ is odd. This concludes the proof.
\end{proof}

\begin{figure}[ht]
\centering
\begin{tikzpicture}[
    scale=0.9,
    vertex/.style={circle,draw,fill=white,inner sep=1.3pt},
    every node/.style={font=\small}
]

% ------------------------------------------------
% (a) T_1 \circ K_{1,n}
% ------------------------------------------------

\node[draw,ellipse,minimum width=2.7cm,minimum height=1.15cm]
    (T1a) at (0,0) {$T_1$};

% left star
\node[vertex] (s1c) at (-1.45,2.0) {};
\node[vertex] (s11) at (-2.05,2.48) {};
\node[vertex] (s12) at (-1.82,2.90) {};
\node[vertex] (s13) at (-1.08,2.90) {};
\node[vertex] (s14) at (-0.85,2.48) {};

\draw (s1c)--(s11);
\draw (s1c)--(s12);
\draw (s1c)--(s13);
\draw (s1c)--(s14);
\node at (-1.45,2.78) {$\cdots$};

% right star
\node[vertex] (s2c) at (1.45,2.0) {};
\node[vertex] (s21) at (0.85,2.48) {};
\node[vertex] (s22) at (1.08,2.90) {};
\node[vertex] (s23) at (1.82,2.90) {};
\node[vertex] (s24) at (2.05,2.48) {};

\draw (s2c)--(s21);
\draw (s2c)--(s22);
\draw (s2c)--(s23);
\draw (s2c)--(s24);
\node at (1.45,2.78) {$\cdots$};

% attachments
\draw[dashed] (-0.75,0.5) -- (s1c);
\draw[dashed] (0.75,0.5) -- (s2c);

\node at (0,2.10) {$\cdots$};
\node at (0,3.20) {$K_{1,n}$-copies};

\node at (0,-0.95) {(a) $T_1\circ K_{1,n}$};

% ------------------------------------------------
% (b) T_1 \circ P_4
% ------------------------------------------------

\begin{scope}[xshift=6.2cm]

\node[draw,ellipse,minimum width=2.7cm,minimum height=1.15cm]
    (T1b) at (0,0) {$T_1$};

% left P4
\node[vertex] (p11) at (-1.45,1.65) {};
\node[vertex] (p12) at (-1.45,2.10) {};
\node[vertex] (p13) at (-1.45,2.55) {};
\node[vertex] (p14) at (-1.45,3.00) {};
\draw (p11)--(p12)--(p13)--(p14);

% right P4
\node[vertex] (p21) at (1.45,1.65) {};
\node[vertex] (p22) at (1.45,2.10) {};
\node[vertex] (p23) at (1.45,2.55) {};
\node[vertex] (p24) at (1.45,3.00) {};
\draw (p21)--(p22)--(p23)--(p24);

% attachments
\draw[dashed] (-0.75,0.5) -- (p11);
\draw[dashed] (0.75,0.5) -- (p21);

\node at (0,2.10) {$\cdots$};
\node at (0,3.45) {$P_4$-copies};

\node at (0,-0.95) {(b) $T_1\circ P_4$, $|V(T_1)|$ odd};

\end{scope}

\end{tikzpicture}

\caption{The two classes of corona products of trees with connected game chromatic number three. The ellipse represents an arbitrary first factor $T_1$, and only representative attached copies are shown.}
\label{fig:three-characterization}
\end{figure}
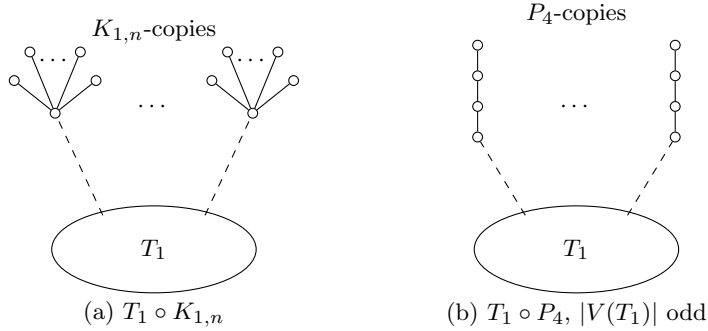

\subsection{Sufficient conditions for four colors}

The upper bound obtained in Theorem~\ref{thm:universal-five} can often be improved to four colors. First, we derive a sufficient structural condition based on the reduced coloring game introduced by \cite{dllrt-2015}.

Let $F$ be a forest, and let $E_{>2}(F)$ denote the set of edges with at least one endpoint of degree greater than $2$. Further, let $E_{>>2}(F)$ denote the set of edges with both endpoints of degree greater than $2$.

For a partially colored forest $F$, recall that a trunk $R$ of $F$ is a maximal connected subgraph of $F$ in which every colored vertex is a pendant vertex of $R$. We denote the collection of all trunks of $F$ by $\mathcal R(F)$. The reduced graph $\mathcal R'(F)$ is obtained from the trunk forest by deleting every edge that does not belong to $E_{>2}(\mathcal R(F))$. Thus, the degrees used in the construction of the reduced graph are the degrees in the current trunk forest rather than those in the original partially colored forest. For an uncolored tree $T$, however, we have $\mathcal R(T)=T$, and hence $\mathcal R'(T)$ is obtained simply by deleting every edge whose endpoints both have degree at most $2$ in $T$. We say that a vertex covers a set of edges if it is incident with every edge in that set.

We will use the $3$-reduced coloring game ($3$-RCG) of \cite{dllrt-2015}. This game is played on the current reduced graph, with Bob moving first and being allowed to pass, while Alice may only color vertices of degree at least $3$. Alice wins as soon as all vertices of degree at least $3$ have been colored.

\begin{lemma}
\label{lem:rcg-mcg}
Let $F$ be a forest. If Alice wins the $3$-RCG on $\mathcal R'(F)$, then she also wins the $3$-MCG on $F$.
\end{lemma}

\begin{proof}
Assume that the $3$-MCG is played on $F$. Alice simulates her winning $3$-RCG strategy on the current reduced graph. After every move, the trunk forest and the corresponding reduced graph are updated.

Every move or pass made by Bob in the MCG corresponds to a legal move or pass in the simulated RCG. Indeed, the reduced graph is obtained from the trunk forest by deleting edges, so a color that is legal for Bob in $F$ is also legal in the reduced graph.

Conversely, every move prescribed for Alice by the RCG strategy is legal in the MCG. Let $u$ be an uncolored vertex of degree at least $3$ in the current reduced graph. Since $u$ has degree at least $3$, every edge incident with $u$ in the trunk forest belongs to $E_{>2}$ and is therefore retained in the reduced graph. Moreover, an uncolored vertex has the same neighbors in the trunk forest as in the original forest $F$. Hence $u$ has the same colored neighbors in the reduced graph as in $F$, and any color that is legal at $u$ in the RCG is also legal at $u$ in the MCG.

When the RCG strategy terminates, every remaining uncolored vertex of $F$ has degree at most $2$. Indeed, if an uncolored vertex $u$ had degree at least $3$ in $F$, then it would have the same degree in the trunk forest. Consequently, every edge incident with $u$ would be retained in the reduced graph, and $u$ would still have degree at least $3$ there, contradicting the termination condition of the RCG.

With three colors available, a vertex of degree at most $2$ can never become uncolorable. Alice may therefore abandon the simulated RCG strategy and complete the $3$-MCG greedily. Hence she wins the $3$-MCG on $F$.
\end{proof}

\begin{proposition}
\label{prop:reduced-four}
Let $T_1$ and $T_2$ be trees. Suppose that every component $R$ of $\mathcal R'(T_2)$ contains a vertex that covers $E_{>>2}(R)$. Then $\chi_{cg}(T_1\circ T_2)\leq 4$.
\end{proposition}

\begin{proof}
Since $T_2$ is initially uncolored, the components of $\mathcal R'(T_2)$ are exactly its reduced trunks, and none of them contains a colored vertex. Therefore the hypothesis is precisely condition~(ii) of Lemma~20 of \cite{dllrt-2015}. Thus, Alice has a winning strategy in the $3$-RCG on $\mathcal R'(T_2)$.

Lemma~\ref{lem:rcg-mcg} therefore yields Alice's win in the $3$-MCG on $T_2$. Applying Lemma~\ref{lem:mcg-transfer}, we obtain $\chi_{cg}(T_1\circ T_2)\leq 4$.
\end{proof}
The above result yields the following simple consequence.

\begin{corollary}
\label{cor:no-adjacent-branch}
Let $T_1$ and $T_2$ be trees. If no two vertices of degree at least $3$ are adjacent in $T_2$, then $\chi_{cg}(T_1\circ T_2)\leq 4$.
If, in addition, $|V(T_1)|\geq2$, $|V(T_2)|\geq2$, and $T_2$ is neither a star nor $P_4$, then $\chi_{cg}(T_1\circ T_2)=4$.
\end{corollary}

\begin{proof}
Suppose that no two vertices of degree at least $3$ are adjacent in $T_2$. Then no edge of $T_2$ has both endpoints of degree at least $3$. Since $\mathcal R'(T_2)$ is a subgraph of $T_2$, every component $R$ of $\mathcal R'(T_2)$ satisfies $E_{>>2}(R)=\emptyset$. Therefore, the assumptions of Proposition~\ref{prop:reduced-four} are trivially satisfied, which gives $\chi_{cg}(T_1\circ T_2)\leq4$.

Under the additional assumptions, Corollary~\ref{cor:nonstar-not-p4} gives $\chi_{cg}(T_1\circ T_2)\geq4$. Hence $\chi_{cg}(T_1\circ T_2)=4$.
\end{proof}
The condition in Proposition~\ref{prop:reduced-four} is only sufficient. It remains useful even when $T_2$ contains arbitrarily many vertices of degree at least $3$, provided that, in each component of $\mathcal R'(T_2)$, all edges of $E_{>>2}$ share a common endpoint. Thus, it is considerably more flexible than conditions that merely bound the total number of branching vertices.

\subsection{When five colors are necessary}

We conclude the section by proving that there exists an infinite family of corona products attaining the universal upper bound $5$.

Let $H_d=\operatorname{cat}(0,2,2,2,2,0)$ be the $14$-vertex tree denoted by $T'$ in Figure~5 of \cite{dllrt-2015}. They proved in Theorem~18 that Bob has a winning strategy in the $3$-ECG on $H_d$. Moreover, $H_d$ is a smallest-order tree with game chromatic number $4$.

\begin{theorem}
\label{thm:Hd-five}
Let $T_1$ and $T_2$ be trees with $|V(T_1)|\geq 2$. If $T_2$ contains $H_d=\operatorname{cat}(0,2,2,2,2,0)$ as a connected subtree, then
$\chi_{cg}(T_1\circ T_2)=5$.
\end{theorem}

\begin{proof}
By Theorem~\ref{thm:universal-five}, we have
$\chi_{cg}(T_1\circ T_2)\leq 5$.

By Theorem~18 of \cite{dllrt-2015}, Bob has a winning strategy in the $3$-ECG on $H_d$. Hence, Lemma~\ref{lem:ecg-transfer} implies that Bob has a winning strategy in the connected coloring game on $T_1\circ T_2$ with four colors.

Since the connected coloring game is not necessarily monotone with respect to the number of available colors, we must also exclude the possibility that Alice wins with three colors. The tree $H_d$ contains $P_5$ as a connected subtree, and by Lemma~9 of \cite{dllrt-2015}, Bob has a winning strategy in the $2$-ECG on $P_5$. Applying Lemma~\ref{lem:ecg-transfer} again, we conclude that Bob also wins the connected coloring game on $T_1\circ T_2$ with three colors. Finally, since $T_1\circ T_2$ contains a triangle, it cannot be properly colored with fewer than three colors.

Thus Alice does not have a winning strategy with fewer than five colors, whereas Theorem~\ref{thm:universal-five} guarantees that she has a winning strategy with five colors. Consequently,
$\chi_{cg}(T_1\circ T_2)=5$.
\end{proof}

The argument above is not specific to $H_d$. More generally, Lemma~\ref{lem:ecg-transfer} shows that Alice loses the connected coloring game with four colors on $T_1\circ T_2$ whenever $T_2$ contains a connected subtree on which Bob wins the $3$-ECG. If such a subtree also contains a connected subtree on which Bob wins the $2$-ECG, then, together with Theorem~\ref{thm:universal-five}, this implies that the connected game chromatic number of $T_1\circ T_2$ is $5$. Thus, any further $3$-ECG obstruction appearing as a connected subtree of $T_2$ can be transferred to the corona product in the same way.

\section{Concluding remarks and open problems}
\label{sec:conclusion}

Corona products provide a natural setting for studying the effect of the connectivity restriction on coloring games. In this paper, we obtained a sharp degree bound for general corona products and exact values for the family $P_m\circ Q_n$, where $Q_n=K_{n,n}-\mathcal M$.

These formulas reveal several unbounded phenomena. Indeed, the choice of the starting player may change the required number of colors by an arbitrarily large amount in either direction. Furthermore, a connected subgraph may have an arbitrarily larger connected game chromatic number than its ambient connected graph. Moreover, the connected game chromatic number of a corona product is not bounded in terms of the connected game chromatic numbers of its factors.

The comparison with the ordinary coloring game also yields interesting results. Namely, we showed that, within corona products, the difference $\chi_g-\chi_{cg}$ can be arbitrarily large. Moreover, our extension of Hollom's example gives infinitely many connected graphs whose game chromatic number is smaller than their connected game chromatic number. However, in all currently known examples with $\chi_{cg}>\chi_g$, the difference is equal to $1$. This motivates the following question:

\begin{question}
Can $\chi_{cg}(G)-\chi_g(G)$ be at least $2$? In particular, does there exist a corona product $G\circ H$ such that $\chi_g(G\circ H)<\chi_{cg}(G\circ H)$?
\end{question}

For corona products of trees, the situation is considerably more structured. We proved the universal bound $\chi_{cg}(T_1\circ T_2)\leq5$ and characterized all corona products with $\chi_{cg}(T_1\circ T_2)=3$ when $|V(T_2)|\geq2$. We also obtained structural conditions guaranteeing the value $4$ and exhibited a broad family of examples attaining the value $5$. The modified and expanded coloring games provide a useful conceptual explanation for this behavior: local strategies in the second factor can be transferred to the corona once the color of the corresponding central vertex is removed from the local palette.

The remaining cases naturally lead to the following problem.

\begin{question}
Characterize all pairs of trees $T_1,T_2$ for which $\chi_{cg}(T_1\circ T_2)=4$ and those for which $\chi_{cg}(T_1\circ T_2)=5$.
\end{question}

Together with the universal upper bound and Theorem~\ref{thm:three-characterization}, such a characterization would complete the determination of $\chi_{cg}(T_1\circ T_2)$ for all corona products of trees with $|V(T_2)|\geq2$.

Our transfer results show that a better understanding of the modified and expanded coloring games on trees would immediately lead to sharper results for corona products.

\begin{question}
Which trees admit a winning strategy for Alice in the $3$-MCG, and which trees admit a winning strategy for Bob in the $3$-ECG?
\end{question}

In particular, the tree $H_d=\operatorname{cat}(0,2,2,2,2,0)$ provides a $3$-ECG obstruction that can be transferred to corona products and forces the value $5$.

\begin{question}
Is there a finite characterization of the minimal connected subtrees whose presence in $T_2$ forces $\chi_{cg}(T_1\circ T_2)=5$ for every tree $T_1$ with $|V(T_1)|\geq2$?
\end{question}

Finally, both the ordinary and the connected coloring games are PSPACE-complete in general \citep{costa-2020,lima-2022}. On the other hand, corona products of trees have a strong structural restriction: five colors always suffice for Alice to win the connected coloring game. This raises the following question.

\begin{question}
What is the computational complexity of the connected coloring game on corona products, and in particular on corona products of trees?
\end{question}

It would also be interesting to study the connected game chromatic number for other standard graph products and to determine which of the phenomena established in this paper persist in those settings.

\backmatter

\section*{Acknowledgements}
%%%%%%%%%%%%%%%%%%%%%%%%%%%%%%%%%%%%%%%%%%%%%%%%%%%%%%%%%%%%%%%%%%%%%

The author acknowledge the financial support from the Slovenian Research Agency (research projects N1-0431).

\end{document}